\documentclass{article}

\usepackage{fullpage}
\usepackage{graphics}
 \usepackage{graphicx}	
\usepackage{amssymb}
 \usepackage{amsmath}
 \usepackage{amsthm}

\usepackage{stmaryrd}
\usepackage{lmodern}

\usepackage{xcolor}
\usepackage{soul}

\newcommand{\tcr}[1]{\textcolor{red}{#1}}
\newcommand{\tcb}[1]{\textcolor{blue}{#1}}

\setstcolor{purple}

\renewcommand{\theequation}{\arabic{section}.\arabic{equation}}

 \newtheorem{theorem}{Theorem}[section]
\newtheorem{lemma}{Lemma}[section]

\newtheorem{corollary}[lemma]{Corollary}

\newcommand{\R}{\mathbb{R}}

\renewcommand{\H}{\mathcal{H}} 

\newcommand{\E}{E_{p,\la}} 
 
\newcommand{\A}{\mathcal{A}}

\newcommand{\lb}{\llbracket} 
\newcommand{\rb}{\rrbracket} 
\newcommand{\lra}{\longrightarrow} 
\newcommand{\Lra}{\Longrightarrow}
\newcommand{\ra}{\rightarrow}

 \newcommand{\sse}{\subseteq}
\renewcommand{\~}{\tilde} 
\renewcommand{\-}{\bar}

\newcommand{\pd}{\partial} 
\newcommand{\rhu}{\;{\overset{*}{\rightharpoonup}}\; }

\newcommand{\fal}{\forall}
\newcommand{\8}{\infty} 

\newcommand{\vph}{\varphi}
\newcommand{\vep}{\varepsilon} 
 
\newcommand{\dt}{\delta} 
\newcommand{\al}{\alpha}
 
\newcommand{\gm}{\gamma}
 
\newcommand{\sm}{\Sigma}
 \newcommand{\om}{\Omega}
\newcommand{\la}{\lambda} 
\renewcommand{\vartheta}{\Theta}

\newcommand{\disp}{\displaystyle}

\DeclareMathOperator{\dist}{dist}

\DeclareMathOperator{\diam}{{\operatorname{diam}}}

{   \end{list} }

\renewcommand{\d}{\,{\operatorname{d}}}

\title{The average distance problem with an Euler elastica penalization\footnote{This paper will appear in Interfaces and Free Boundaries. }}

\begin{document}

\date{}

\author{Qiang Du \thanks{Department of Applied Physics and Applied Mathematics and Data Science Institute, Columbia University, New York, NY, 10027, USA.  Email: qd2125@columbia.edu. }
\thanks{Supported in part by NSF DMS-2012562 and DMS-1937254.}
\and Xinyang Lu \thanks{Corresponding author. Department of Mathematical Sciences, Lakehead University, Thunder Bay, Ontario, P7B 5E1, Canada AND
Department of Mathematics and Statistics, McGill University,
805 Sherbrooke St. W., Montreal, QC H3A 0B9, Canada. Email: xlu8@lakeheadu.ca. }
\thanks{Supported in part by NSERC Discovery Grant and internal Lakehead University grants. }
\and Chong Wang \thanks{Department of Mathematics, Washington and Lee University, Lexington, VA, 24450, USA. Email: cwang@wlu.edu.}
}

\maketitle

\begin{abstract}
\noindent 
We consider the minimization of 
an  average distance functional defined on a two-dimensional domain $\Omega$ 
with an Euler elastica penalization associated with $\pd \Omega$, the boundary of $\Omega$.
The average distance is given by
\begin{equation*}
\int_{\Omega} \dist^p(x,\pd \Omega )\d x
\end{equation*}
where 
$p\geq 1$ is a given parameter,
and $\dist(x,\pd \Omega)$ is the Hausdorff distance between $\{x\}$ and $\pd \Omega$.
The penalty term is a multiple of the Euler elastica (i.e., the Helfrich bending energy or the Willmore energy) of
the boundary curve ${\pd \Omega}$, 
 which is proportional to the integrated squared curvature defined on $\pd \Omega$, as given by
\begin{equation*}
\la \int_{\pd \Omega} \kappa_{\pd \Omega}^2\d\H_{\llcorner \pd \Omega}^1,
\end{equation*}
where $\kappa_{\pd \Omega}$ denotes the (signed) curvature of $\pd \Omega$ and $\la>0$ denotes a
penalty constant. 
The domain
 $\Omega$ is allowed to vary among compact, convex sets
of $\mathbb{R}^2$ with Hausdorff dimension equal to $2$\tcr{.} 
Under 
no a priori assumptions on the regularity of the boundary $\pd \Omega$, we prove the
existence of
minimizers of $E_{p,\la}$. Moreover, we establish the
$C^{1,1}$-regularity of its minimizers. An original construction of a suitable family of competitors plays a decisive role in proving the regularity.
 \end{abstract}


\textbf{Classification. }
49Q20, 
 49K10, 
 49Q10 


\section{Introduction}

The curvature of boundaries plays an important role in many physical and biological models. For instance, the elasticity of cell membranes 
is strongly correlated to its bending, and thus to its curvature.
One way to quantify the bending energy per unit area
of closed lipid bilayers was proposed by Helfrich in \cite{Helf}, and is now commonly referred to as ``Helfrich bending energy''.
A related notion, from differential geometry, is the ``Willmore energy'', which measures how much a surface differs
from the sphere \cite{DLW05}.
In 2D, the Willmore energy simplifies to  be a multiple of the integrated squared curvature, which is also commonly referred as the
Euler elastica.

Easy access to the boundary is also relevant in nature: many processes
 such as heat dissipation, waste disposal and nutrient absorption,
  are more efficient when the whole body has ``easy access'' to its boundary. One way to quantify the ``average accessibility''
  for points of a set $\Omega\subset \mathbb{R}^2$ to the boundary $\partial \Omega$
 is an energy functional of the form
 \begin{equation}\label{FF}
\Omega\longmapsto  \int_{\Omega} \dist^p(x,\pd \Omega)\d x\tcb{,}
\end{equation}
for a given parameter $p\geq 1$.

There are other energy functionals  sharing similar geometric features with \eqref{FF}.
For example,  \eqref{FF} is 
 formally similar to the average-distance functional associated with a  {\em given} domain 
$\Omega\subset \mathbb{R}^2$, 
\begin{equation}\label{FA}
\Sigma\longmapsto \int_{\Omega} \dist^p(x,\sm)\d x,
\end{equation}
where the unknown $\sm$ varies among compact subsets of $\bar{\Omega}$. In many existing studies, 
 $\sm$ is assumed  to be a connected set with its Hausdorff dimension 
equal to $1$ and its one dimensional Hausdorff measure is to be bounded from above by  a specified constant.
Problems of this type are used in many modeling applications, such as urban planning and optimal pricing. 
For a (non-exhaustive) list of references on the average-distance problem we refer to the works by Buttazzo et al. 
\cite{BMS, BOS, BS3, BS4} and \cite[Chapter 3.3]{BPSS}.
Also related are the papers by Paolini and Stepanov \cite{PaoSte},
Santambrogio and Tilli \cite{ST}, Tilli \cite{Til}, Lemenant and Mainini \cite{LM},
Slep\v{c}ev \cite{Slepcev}, and the review paper by Lemenant \cite{Lem}. Meanwhile, if  $\sm$ 
is assumed to vary among  sets $\Omega$ consisting of discrete points with a fixed cardinality, say $k$, 
 then the  minimization of
the functional  in \eqref{FA}, often named the quantization error in this case, is related to the centroidal Voronoi
tessellations \cite{DFG} and $k$-means, which are widely studied in  subjects such as
vector quantization, signal compression, sensor and resource placement, geometric meshing, and so  on \cite{DGJ}.
Similar variational problems entailing a competition between classical perimeter and nonlocal repulsive interaction were studied by Muratov and Kn\"upfer \cite{mk},
 Goldman, Novaga and Ruffini \cite{gnru}, and Goldman, Novaga and R\"oger \cite{gnro}.  Figalli, Fusco, Maggi, Millot, and Morrini studied a competition between a nonlocal $s$-perimeter and a nonlocal repulsive interaction term \cite{ffmmm}.

In this work, we consider the average distance energy functional as a functional
of the domain $\Omega$ with $\sm=\pd\Omega$  and penalized by the Euler elastica  of $\pd\Omega$, as
given by
\begin{eqnarray*}
E_{p,\lambda} (\Omega) = \int_{\Omega}  \text{ dist}^p (x, \partial \Omega) \ dx + \lambda \int_{\partial \Omega} \kappa_{\partial \Omega}^2 d {\cal{H}}_{\llcorner \partial \Omega}^1,
\end{eqnarray*}
where $p \geq 1, \lambda > 0$ are given parameters,  with $\lambda$ proportional to a bending constant,
and $\H_{\llcorner \pd \Omega}^1$ denotes the Hausdorff measure restricted on $\pd \Omega$. For further properties of the Hausdorff measure, we refer to \cite{AFP}. 
We consider a free boundary problem associated with 
the minimization of $E_{p,\lambda}$ among domains $\Omega$ in 
 the following admissible set
\begin{eqnarray*}
{\cal {A}}  := \{ \Omega :  \Omega \subset \mathbb{R}^2  \text{ is compact, convex and Hausdorff two-dimensional} \}.
\end{eqnarray*}
For any $\Omega_1, \Omega_2 \in {\cal {A}}$, define the metric in $\cal{A}$ as
\begin{eqnarray} \label{d}
d(\Omega_1, \Omega_2) := {\cal{H}}^2 (\Omega_1 \triangle \Omega_2),
\end{eqnarray}
where $\triangle$ denotes the symmetric difference of the two sets and ${\cal{H}}^2 $ denotes the two-dimensional
Hausdorff measure.

The term 
\begin{equation}\label{K1}
\int_{\pd \Omega} \kappa_{\pd \Omega}^2\d\H_{\llcorner \pd \Omega}^1 \nonumber
\end{equation}
is the integrated squared curvature \cite{dC76}. Since we do not make any a priori assumptions on the regularity of the boundary
$\pd \Omega$, we need to make sense of the integrand $\kappa_{\pd \Omega}$.
 For future reference we will define it as follows:
 let $\gm$ be an arc-length parameterization of $\pd  \Omega$,
 and define
\begin{equation}\label{K2}
\int_{\pd \Omega} \kappa_{\pd \Omega}^2\d\H_{\llcorner \pd \Omega}^1 :=
\left\{\begin{array}{cl}
{\disp{ \int_0^{\H^1(\pd \Omega)} |\gm''|^2
\d s}}  & \text{if } \gm \in H^2((0,\H^1(\pd \Omega));\R^2 ), \\
&\\
+\8 & \text{otherwise}.
\end{array}\right.
\end{equation}
Here,  $\H^1(\pd \Omega)$ denotes the total length of $\pd \Omega$.
That is, we are reducing our minimization problem to quite regular sets,
i.e. domains $\om$ whose boundaries admit an $H^2$ regular arc-length parameterization. Therefore, we are considering the minimization problem
\begin{align*}
\inf \bigg\{  &\int_{\Omega}  \text{ dist}^p (x, \partial \Omega) \d x + \lambda\int_0^{\H^1(\pd \Omega)} |\gm''|^2
\d s : \H^1(\pd \Omega)>0,\\
& |\gm'|=1,\ \gm( [ 0,\H^1(\pd \Omega) ]) \text{ is the boundary of a compact convex set} \bigg\}.
\end{align*}

We note first a simple rescaling analysis where the domain $\Omega$ is stretched by a factor
$\epsilon>0$. Given the two-dimensional nature, if $\epsilon>1$,  then the average distance functional is  scaled by no more  than $\epsilon^{2+p}$ but no less than $\epsilon^2$. Meanwhile, the Euler elastica  gets scaled by $1/\epsilon$.  This shows that the optimal $\Omega$, if exists, must have a suitable and finite size for any prescribed  $\lambda>0$. Indeed, the energy
considered might be viewed as a competition between access to the boundary and the elastic stiffness of the boundary.

The main result of this paper is:
\begin{theorem}\label{regularity}
 Given $p \geq 1, \lambda > 0$, any minimizer $\Omega$ of $E_{p, \lambda}$ is $C^{1,1}$-regular with a Lipschitz constant at most
\begin{eqnarray*}
C = C (p, \lambda) : = \sqrt{ \la^{-1}p(C_1+1)^{p-1}\pi C_1^2+ 2 C_2},
\end{eqnarray*}
where 
\begin{eqnarray}
&& C_1 = C_1 (p, \lambda)  :=   (p+1)(p+2) \left   (\frac{24}{\la} \right)^{p+1} (2 (1+\pi \lambda))^{p+2},  \label{C_1}\\
&& C_2 = C_2 (p, \lambda)  :=   32 (\la^{-1}+ \pi)^2 + 32 \sqrt{2 C_1 \pi }  (\la^{-1} + \pi)^{5/2},
 \label{C_2}
\end{eqnarray}
are constants independent of $\Omega$.
That is, 
the boundary $\pd \Omega$ admits a $C^{1,1}$-regular, arc-length parameterization $\gm:[0,\H^1(\pd \Omega)]\lra \R^2$
such that
\begin{equation*}
|\gm'(t_1)-\gm'(t_2)|\leq C |t_1-t_2|
\end{equation*}
 for any $t_1,t_2$.
\end{theorem}

The rest of the paper is organized as follows: Section \ref{sec:estimate} is dedicated to proving some auxiliary estimates on elements of minimizing sequences.  Existence of minimizers is shown in Section \ref{sec:existence}, while $C^{1,1}$-regularity is  established in Section \ref{sec:proof}. Finally, in Section \ref{sec:conclusion}, we explore several future directions to further our understanding of the penalized average distance problem. Technical results concerning properties of convex
sets used in this paper will be presented in the Appendix.

\section{Estimates}\label{sec:estimate}

This section is dedicated to establishing quantitative bounds on the diameter and the area of any domain associated with the minimizing sequences of $\E$. In particular, the main result is Lemma \ref{p}, which provides a uniform upper bound on the diameter, crucial to the proof of the
existence of minimizers.

\textbf{Remark:} It is worth noting that, due to \eqref{K2}, any set $\Omega$ whose boundary is not $C^1$-regular will have infinite energy, since a corner  on $\partial\Omega$ with a discontinuous tangent  corresponds to the Dirac measure in the curvature measure $\kappa_{\partial\Omega}$. Thus we can restrict ourselves to $C^1$-regular sets.

\begin{lemma}\label{diam}
Given $p\geq 1$, $\la>0$, for any $\Omega \in \cal A$, it holds that
\begin{equation}\label{diam-low}
\diam(\Omega) \geq \frac{4\pi\la}{ \E(\Omega)}.
\end{equation}
Then, for any
 minimizing sequence $\Omega_n\sse \cal A $ (that is, $\E(\Omega_n)\to \inf_{ \cal A } \E$),  it holds that
\begin{equation}\label{diam-l}
\diam(\Omega_n) \geq \frac{2\pi\la}{1+\pi\la }
\end{equation}
for all sufficiently large $n$.
\end{lemma}

\begin{proof}
Consider an arbitrary $\Omega \in \cal{A}$. Choose $x,y\in \pd \Omega$ such that
$|x-y|=\diam (\Omega)$. Note that
$\Omega \sse B(x,\diam(\Omega))$, hence due to the convexity of $\Omega$
(see Lemma \ref{perimeter vs diameter convex sets}), it follows that
\begin{equation*}
\H^1(\pd \Omega)\leq \pi \diam (\Omega).
\end{equation*}
As $\pd \Omega$ is a closed convex curve with winding number equal to $1$, and our restriction on the curvature term ensures that the boundary
is $H^2$ regular,
it follows that 
$$\int_{\pd \Omega} |\kappa_{\pd \Omega}|\d\H_{\llcorner \pd \Omega}^1 =2\pi,$$
and by H\"older's inequality it holds that
\begin{equation*}
\E(\Omega)\geq 
\la\int_{\pd \Omega} \kappa_{\pd \Omega}^2\d\H_{\llcorner \pd \Omega}^1 \geq \frac{4\pi^2\la}{\H^1(\pd \Omega)}  \geq \frac{4\pi\la}{ \diam(\Omega)},
\end{equation*}
hence \eqref{diam-low}.

\medskip

To prove \eqref{diam-l}, we show first that $\inf_{ \cal A} \E<+\8$. Consider the unit ball $B:=B\big((0,0),1\big)$, and note that 
\begin{align} \label{inf E}
\inf_{  \cal A } \E &\leq \E(B)=\int_B \dist^p(x,\pd B)\d x+\la \int_{\pd B}\kappa_{\pd B}^2\d\H^1_{\llcorner \pd B}
\leq \frac{\pi}{3}+2\pi\la.
\end{align}
Let $\Omega_n\sse   \cal A$ be an arbitrary minimizing sequence. Clearly, since $\E(\Omega_n)\to \inf_{ \cal A } \E$, for all sufficiently large $n$,
it holds that 
\begin{eqnarray} \label{inf E2}
\E(\Omega_n)\leq \inf_{\bar{\cal A}}\E+2 - \frac{\pi}{3}  \overset{\eqref{inf E}}\leq 2+2\pi\la,
\end{eqnarray}
and \eqref{diam-low} gives
$\diam(\Omega_n) \geq \dfrac{2\pi\la}{ 1+\pi\la}$, hence \eqref{diam-l}.

\end{proof}

In the following, we will use the definition of the \emph{total variation} of a function $u$, which is defined as follows.
Let $\Omega \subset \mathbb{R}^n$ be an open set and let $u \in L^1(\Omega)$, then
\begin{eqnarray*}
|| u||_{TV} = \sup \left \{ \int_{\Omega} u \; \text{div}  \phi \;  dx :\phi \in C_{c}^{1} (\Omega; \mathbb{R}^n), || \phi ||_{L^{\infty}(\Omega) } \leq 1 \right \},
\end{eqnarray*}
and
\begin{eqnarray*}
||u||_{BV} = ||u||_{L^1} + ||u||_{TV}.
\end{eqnarray*}

\begin{lemma}\label{a}
Given $p\geq 1$, $\la>0$, and $\Omega \in { \cal A} $, it holds that
\begin{equation}\label{a-low}
\H^2(\Omega)\geq \frac{\pi\la^2}{2\E(\Omega)^2}.
\end{equation}
Moreover, given a minimizing sequence $\Omega_n\sse { \cal A}$ (that is, $\E(\Omega_n)\to \inf_{ \cal A } \E$), 
we have
\begin{equation}\label{a-l}
\H^2(\Omega_n) \geq  \frac{\pi\la^2}{8(1+\pi\la)^2} 
\end{equation}
for all sufficiently large $n$.
\end{lemma}

To simplify notations, for future reference, given a point $z\in \R^2$,  we let
 $z_x$ (resp. $z_y$) denote the $x$ (resp. $y$) coordinate of $z$.
And given points $x,y\in \R^2$, 
we denote by
$$\lb x,y\rb:=\{(1-s)x+sy:s\in[0,1]\}$$
the line segment between $x$ and $y$. 

\begin{proof}
Consider an arbitrary $\Omega \in \cal{A}$.
Choose arbitrary points $\-x,\-y\in \pd \Omega$ such that $|\-x-\-y|=\diam(\Omega)$. Endow $\R^2$
with a Cartesian coordinate system, with the origin at the midpoint $(\-x+\-y)/2$, such that
$$\-x=(-\diam(\Omega)/2,0), \quad \-y=(\diam(\Omega)/2,0).$$
Let $\gm:[0,\H^1(\pd \Omega)]\lra \pd \Omega$ be an arc-length parameterization, and without loss of generality
we impose $\gm(0)=\-x$.
We make and prove the following claims.
\begin{itemize}
\item  $\gm'(0)_x= 0$.\\
Assume the opposite, i.e., $\gm'(0)_x\neq 0$. For $|\vep|\ll 1$, since $\gm$ is $C^1$-regular, 
 it holds that
$\gm(\vep)=\-x+\vep\gm'(0)+v_\vep$, for some vector $v_\vep$ with $|v_\vep|=o(\vep)$ as $\vep\to 0$. Since
$\-y-\-x$ is parallel to the $x$-axis, it follows that
\begin{equation*}
\frac\d{\d t}|\-y-\gm(t)|\bigg|_{t=0} = \lim_{\vep\to 0}\frac{|\-y-(\-x+\vep\gm'(0))|-|\-y-\-x|+o(\vep)}\vep
=\gm'(0)_x\neq 0,
\end{equation*}
hence $t=0$ is not a maximum for $t\mapsto |\-y-\gm(t)|$. This contradicts 
$$|\-y-\-x|=\diam(\Omega)=\max_{x\in \pd \Omega}|\-y-x|,$$
 and the claim is proven.
\end{itemize}
\begin{center}
\begin{figure}[h]
\centering
\includegraphics[scale=0.7]{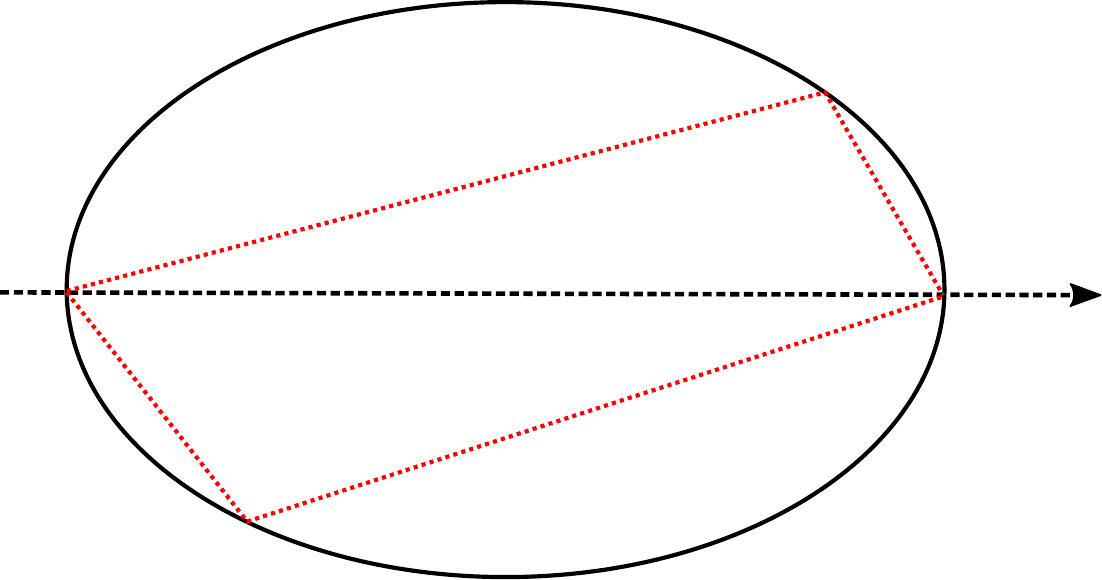}
\put(-217,48){$\bar{x}$}
\put(-30,48){$\bar{y}$}
\put(-62,105){$\gm(t_1)$}
\put(-198,2){$\gm(t_2)$}
\put(-198,105){\large{$\pd \Omega$}}
\put(0,48){$x$-axis}
\footnotesize{\caption{Schematic representation of the construction.}}
\end{figure}
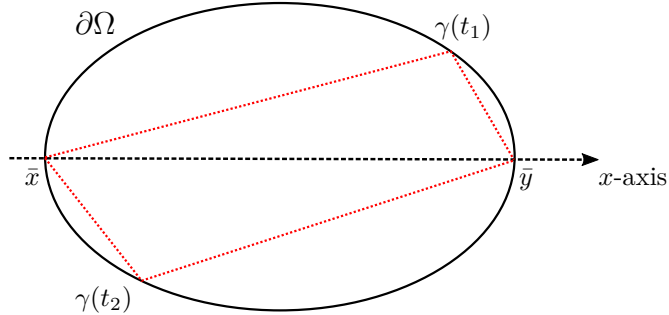
\end{center}
Without loss of generality, we can further impose $\gm'(0)=(0,1)$.
Consider the region $ \Omega \cap \{y\geq 0\}$. Set
\begin{equation*}
t_1:=\inf\{t: \gm'(t)_y=1/2\},
\end{equation*}
where $\gm(t)_y$ denotes the $y$-coordinate of $\gm(t)$. By H\"older's inequality, it follows that
\begin{align}
 \frac{1}{4 t_1} \leq \frac{\|\gm_y'\|^2_{TV(0,t_1)}}{t_1}\leq \int_{\pd \Omega}\kappa_{\pd \Omega}^2\d\H^1_{\llcorner \pd \Omega}\leq \frac{\E(\Omega)}\la
\Lra t_1\geq \frac{\la}{4\E(\Omega)}.\label{inf t0}
\end{align}
Since $1/2\leq \gm_y'(t)\leq 1$ for any $t\in[0,t_1]$, it holds that $\gm(t_1)_y\geq t_1/2$. Due to the convexity
of $ \Omega \cap \{y\geq 0\}$, 
both line segments 
 $\lb \gm(t_1), \- x \rb $ 
and $\lb\gm(t_1),\-y\rb$ 
are contained in $\Omega$,
hence $\triangle \-x\gm(t_1)\-y\sse \Omega$. By construction, the triangle $\triangle \-x\gm(t_1)\-y$ has base 
$\lb \-x,\-y\rb$ and height $\lb \gm(t_1) , (\gm(t_1)_x,0)\rb $, hence
\begin{equation}\label{tr+}
\H^2(\triangle \-x\gm(t_1)\-y)=\frac12| \-x-\-y|\cdot|\gm(t_1)_y|\geq \frac{\diam(\Omega)t_1}{4}.
\end{equation}
By repeating the same construction for $ \Omega \cap \{y\leq 0\}$, we get the existence of $t_2 \geq \frac{\la}{4\E(\Omega)}$ such that the triangle
$\triangle \-x\gm(t_2)\-y$ satisfies
\begin{equation}\label{tr-}
\H^2(\triangle \-x\gm(t_2)\-y)=\frac12| \-x-\-y|\cdot|\gm(t_2)_y|\geq \frac{\diam(\Omega)t_2}4.
\end{equation}
Combining \eqref{tr+} and \eqref{tr-} gives
\begin{equation*}
\H^2(\Omega)\geq \frac{\diam(\Omega)t_1}{4}+\frac{\diam(\Omega)t_2}{4}\overset{\eqref{diam-low},\eqref{inf t0}}\geq \frac{\pi\la^2}{2 \E(\Omega)^2},
\end{equation*}
hence \eqref{a-low}.

To prove \eqref{a-l}, note that the above arguments give
\begin{equation*}
\H^2(\Omega_n)\overset{\eqref{a-low}}\geq \frac{\pi\la^2}{2 \E(\Omega_n)^2}
\overset{\eqref{inf E2}}\geq \frac{\pi\la^2}{8 (1+\pi\la)^2},
\end{equation*}
for any sufficiently large $n$, and proof of \eqref{a-l} is complete.

\end{proof}

\begin{lemma}\label{p}
Given $p\geq 1$, $\la>0$, for any $\Omega \in  \cal A $ it holds that
\begin{equation}\label{diam-high}
\diam(\Omega) \leq  (p+1)(p+2) \left   (\frac{24}{\la} \right)^{p+1} \E(\Omega)^{p+2}  .
\end{equation}
Moreover, for any minimizing sequence $\Omega_n\sse \cal A $ (that is, $\E(\Omega_n)\to \inf_{ \cal A } \E$) it holds that
\begin{equation}\label{diam-h}
\diam(\Omega_n) \leq   C_1
\end{equation}
for all sufficiently large $n$, with $C_1$ defined in \eqref{C_1}.

\end{lemma}

\begin{proof}
Similar to the proof of Lemma \ref{a}, consider an arbitrary $\Omega \in \cal A$, 
and choose arbitrary points $\-x,\-y\in \pd \Omega$ such that $|\-x-\-y|=\diam(\Omega)$. Endow $\R^2$ again
with a Cartesian coordinate system, with the origin at the midpoint $(\-x+\-y)/2$, such that
$$\-x=(-\diam(\Omega)/2,0), \quad  \-y=(\diam(\Omega)/2,0).$$ 

In the proof of Lemma \ref{a} we have shown the existence of a point $q\in \pd \Omega$ (e.g., the point $\gm(t_1)$) such that 
\begin{equation}\label{inf qy}
\triangle \-xq\-y\sse \Omega,\qquad |q_y|\geq \frac{\la}{8\E(\Omega)}.
\end{equation}
\begin{center}
\begin{figure}[h]
\centering
\includegraphics[scale=0.7]{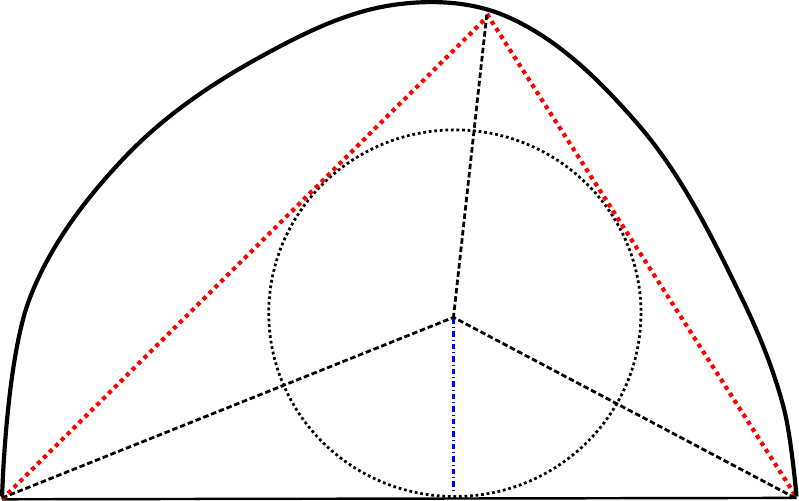}
\put(-168,-5){$\bar{x}$}
\put(1,-5){$\bar{y}$}
\put(-60,105){$q$}
\put(-78,-10){$q_c^\bot$}
\put(-80,40){$q_c$}
\put(-155,68){\large{$\pd \Omega$}}
\footnotesize{\caption{Schematic representation of the construction. Here represented only the region $ \Omega \cap \{y\ge 0\}$. Notice that the sides of $\triangle \bar{x} q \bar{y}$ are tangents of the incircle.}}
\end{figure}
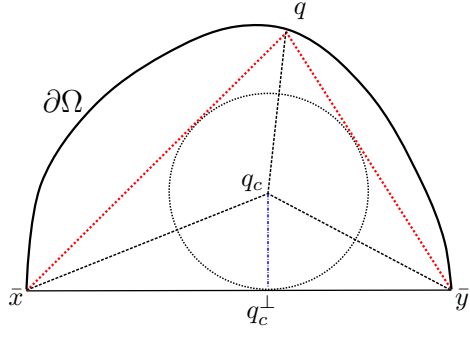
\end{center}
 Let $q_c$ be the incenter 
 of $\triangle \-xq\-y$, and note that for any $z\in  \triangle \-xq_c\-y$
 it holds that 
 $$\dist(z,\pd (\triangle \-xq\-y))=\dist(z,\lb\-x,\-y \rb).$$
 Denote by $q_c^\bot\in \lb\-x,\-y \rb$ the projection of $q_c$ on $\lb\-x,\-y \rb$, and
set
\begin{equation*}
D_1:=|\-x-q_c^\bot|,\qquad D_2:=|\-y-q_c^\bot|,\qquad r:=|q_c-q_c^\bot|.
\end{equation*}
Clearly, $D_1+D_2=\diam(\Omega)$, and
direct computation gives
\begin{align}
\int_{\Omega} \dist^p(z,\pd \Omega)\d z&\geq \int_{\triangle \-x q_c \-y} \dist^p(z,\pd \Omega)\d z
\geq\int_{\triangle \-x q_c \-y} \dist^p(z,\lb\-x,\-y \rb)\d z\notag\\
&=\int_{\triangle \-x q_c \-y} z_y^p\d z
=\int_0^{D_1}\int_0^{\frac{r}{D_1}x} y^p \d y\d x+\int_0^{D_2}\int_0^{\frac{r}{D_2}x} y^p \d y\d x\notag\\
&=\frac{r^{p+1} (D_1+D_2)}{(p+1)(p+2)} = \frac{r^{p+1} \diam(\Omega) }{(p+1)(p+2)}. \label{rD}
\end{align}
To estimate $r$, note that the sides $\lb \-x,q_c\rb$ and $\lb \-y,q_c\rb$ satisfy
$$|\-x-\-y|=\diam(\Omega)\geq \max \{ |\-x-q_c|, |\-y-q_c|\}.$$
Since $q_c$ is the incenter of $\triangle \bar{x} q \bar{y}$ and the sides of $\triangle \bar{x} q \bar{y}$ are tangents of the incircle, we have
\begin{equation*}
\H^2(\triangle \-x q \-y) =\frac12\diam(\Omega)|q_y|=\frac12 (\diam(\Omega)+|\-x-q_c|+ |\-y-q_c|)r \tcr{.}
\end{equation*}
Thus, we infer 
\begin{equation*}
r\geq \frac{|q_y|}{3}\overset{\eqref{inf qy}}\geq  \frac{\la}{24\E(\Omega)}.
\end{equation*}
Plugging into \eqref{rD} gives
\begin{equation*}
 \left ( \frac{\la}{24\E(\Omega)} \right )^{p+1}  \cdot \frac{  \diam(\Omega) }{ (p+1)(p+2)}\leq \int_{\Omega} \dist^p(z,\pd \Omega)\d z\leq \E(\Omega),
\end{equation*}
hence \eqref{diam-high}.

To prove \eqref{diam-h}, note that for any minimizing sequence it holds that
\begin{equation*}
\E(\Omega_n)\overset{\eqref{inf E2}}\leq 2(1+\pi\la)\Lra
\diam(\Omega_n) \overset{\eqref{diam-high}}\leq (p+1)(p+2) \left   (\frac{24}{\la} \right)^{p+1} (2 (1+\pi \lambda))^{p+2} = C_1
\end{equation*}
for all sufficiently large $n$.


\end{proof}

\section{Proof of existence} \label{sec:existence}

Set
$$\-{\cal A}:= \text{completion of } {\cal A} \text{ with respect to} \text{ the metric } d,$$ 
where $d$ is defined in \eqref{d}.

\begin{lemma}\label{conv}
Given a compact set $\Sigma \sse\R^2$, and a sequence of curves $\{\gm_k\}:[0,1]\lra \Sigma$ satisfying
$$\sup_k \|\gm_k'\|_{BV}<+\8, \qquad \sup_{k}\H^1(\gm_k([0,1]))<+\8, $$
where $\|\cdot\|_{BV}$ denotes the BV norm,
then there exists a curve $\gm:[0,1]\lra \Sigma$, such that (upon subsequence) it holds that:
\begin{enumerate}
\item $\gm_k\ra \gm$ in $C^\al$ for any $\al\in[0,1)$,
\item $\gm_k'\ra \gm'$ in $L^p$ for any $p\in[1,\8)$,
\item $\gm_k'' \rhu \gm'' $ in the space of signed Borel measures.
\end{enumerate}
\end{lemma}
The is a classical result (see for instance \cite{Slepcev}, to which we refer for the proof).

\begin{lemma} \label{outA}
If a minimizing sequence $  \om_n  \sse {\cal A}$ converges to some $\om \in \bar{{\cal A}}\backslash {\cal A}$, then $\om$
must be either be a point or a line segment. 
\end{lemma}

\begin{proof}
The compactness of $\om$ can be guaranteed by Lemma \ref{p}. To see that $\om \in \bar{{\cal A}}\backslash {\cal A}$ is convex, let us
consider an arbitrary pair of points $P,Q\in \om$, $t\in (0,1)$, we now show that $(1-t)P+tQ\in \om$.
	Consider sequences $P_n,Q_n\in \om_n$ such that $P_n\to P$, $Q_n\to Q$: since
	each $\om_n$ is convex, $(1-t)P_n+tQ_n\in \om_n$. By Lemma \ref{conv},
	we know $\|\vph_n-\vph\|_{C^0([0,1];\R^2)} \to 0$. As a consequence, 
	\[ d_\H (\pd\om_n,\pd\om)\to 0 \]
	too\tcr{.}
	This allows us to choose, for each $n$, another point $z_n\in \om$ such that
	$|z_n- (   (1-t)P_n+tQ_n ) | \le  d_\H (\pd\om_n,\pd\om)$. By construction, now
	the sequences
	$(1-t)P_n+tQ_n$ and $z_n$ have the same limit. As
	$(1-t)P_n+tQ_n \to (1-t)P+tQ$, and $z_n\to z$, hence $z=(1-t)P+tQ$, using the compactness
	of $\om$ finally gives $z\in \om$. 
Then if $\om$ contains non collinear points $x,y,z$, by convexity $\triangle xyz\sse \om$, which would
 give the contradiction $\om\in {\cal A}$.
\end{proof}


\begin{lemma}\label{convergence of energies implies convergence in metric}
        Given a sequence $\Omega_n\sse \mathcal{A}$, such that $E_{p,\la}(\Omega_n) $ is bounded,
	then there exists $\Omega\in \mathcal{A}$ such that a subsequence of $\Omega_n$ (still denoted by $\Omega_n$) converges to $\Omega$ with respect to the metric $d$ defined in \eqref{d}.
\end{lemma} 

 \begin{proof}
	By \eqref{diam-low} and \eqref{diam-high}, we have
	\[(p+1)(p+2) \left(\frac{24}{\la} \right)^{p+1} \E(\Omega_n)^{p+2}\ge
	\diam(\Omega_n) \geq \frac{4\pi\la}{ \E(\Omega_n)} , \]
	hence $\sup_n \diam(\Omega_n)<+\8$, i.e. $\Omega_n$
	has uniform bounded diameters. Note also that $\sup_n E_{p,\la}(\Omega_n)<+\8 $
	implies
	 \[\sup_n  \int_{\pd \Omega_n} \kappa_{\pd \Omega_n}^2\d\H_{\llcorner \pd \Omega_n}^1 <+\8,\]
	 i.e. the curvatures of $\Omega_n$ are uniformly square integrable. Then, by
	 letting $\gm_n:[0,2\pi]\lra \R^2$ being constant speed parameterizations of $\pd \Omega_n$,
	 we have that a subsequence $\gm_n$ satisfies all the hypotheses of Lemma \ref{conv},
	 concluding the proof.
\end{proof} 

Based on Lemma \ref{diam}, Lemma \ref{a}, Lemma \ref{p} and Lemma \ref{convergence of energies implies convergence in metric}, we get the following corollary.

\begin{corollary} \label{corOmega}
Any minimizer (if they exist at all) satisfies estimates \eqref{diam-l}, \eqref{a-l} and \eqref{diam-h}.
\end{corollary}


\begin{lemma} \label{existence}
For any $p \geq 1, \lambda > 0$, the functional $E_{p, \lambda}$ admits a minimizer in $\cal{A}$.
\end{lemma} 

\begin{proof}
Consider a minimizing sequence $\Omega_n\sse {\cal A}$. 
Since $\E$ is invariant under rigid movements, we can assume that $(0,0) \in \Omega_n $ for any $n$.
In view of \eqref{inf E2}, without loss of generality, we can also impose
\begin{equation*}
\sup_n \E(\Omega_n)\leq 2(1+\pi\la).
\end{equation*}
Then by Lemma \ref{p} we get $\sup_n \diam(\Omega_n)\leq C_1$, hence
$$ \Omega_n\sse B((0,0),C_1) \quad \text{for any }n.$$
Thus $\Omega_n$ is a sequence of uniformly bounded, compact sets, and there exists (upon subsequence, which we do not relabel)
a limit set $\Omega\in \bar{{\cal A}}$ such that $\Omega_n\to \Omega$ in the metric $d$ (defined in \eqref{d}). 

\medskip

We claim
\begin{align}
\int_\Omega \dist^p(z,\pd \Omega)\d z&=\lim_{n\to+\8} \int_{\Omega_n}\dist^p(z,\pd \Omega_n)\d z,\label{con3} \\
\int_{\pd \Omega}\kappa_{\pd \Omega}^2\d\H^1_{\llcorner \pd \Omega}&\leq
\liminf_{n\to+\8}\int_{\pd \Omega_n}\kappa_{\pd \Omega_n}^2\d\H^1_{\llcorner \pd \Omega_n} \label{lsc k}.
\end{align}
The latter, i.e. \eqref{lsc k}, follows rather straightforwardly from the lower semicontinuity
of the $H^2$ norm. 

\medskip

We need to prove
 \eqref{con3}. In the following, it is useful to recall Lemma \ref{perimeter vs diameter convex sets}, which states that the diameter is continuous with respect to the convergence
in $\A$.
We split the sums 
\begin{align*}
\int_{\Omega_n} \dist^p(z,\pd \Omega_n)\d z & =\int_{\Omega_n\backslash \Omega} \dist^p(z,\pd \Omega_n)\d z+\int_{\Omega_n\cap \Omega} \dist^p(z,\pd \Omega_n)\d z,\\
\int_{\Omega} \dist^p(z,\pd \Omega)\d z & =\int_{ \Omega \backslash \Omega_n} \dist^p(z,\pd \Omega)\d z+\int_{\Omega_n\cap \Omega} \dist^p(z,\pd \Omega)\d z,
\end{align*}
and note that
\begin{align}
\bigg|\int_{\Omega_n}& \dist^p(z,\pd \Omega_n)\d z-\int_{\Omega} \dist^p(z,\pd \Omega)\d z \bigg| \notag\\
&\leq \int_{\Omega_n\backslash \Omega} \dist^p(z,\pd \Omega_n)\d z+
\int_{\Omega \backslash \Omega_n} \dist^p(z,\pd \Omega)\d z \label{en-1}\\
&+\int_{\Omega_n\cap \Omega}| \dist^p(z,\pd \Omega_n)-\dist^p(z,\pd \Omega)|\d z.\label{en-2}
\end{align}
Moreover,
\begin{align*}
\int_{\Omega_n\backslash \Omega} \dist^p(z,\pd \Omega_n)\d z & \leq \H^2(\Omega_n\backslash \Omega)
\diam (\Omega_n)^p \leq \H^2(\Omega_n\backslash \Omega) C_1^p \to 0,\\
\int_{\Omega \backslash \Omega_n} \dist^p(z,\pd \Omega)\d z &\leq \H^2(\Omega \backslash \Omega_n)\diam (\Omega)^p
\leq \H^2(\Omega \backslash \Omega_n) C_1^p \to 0,
\end{align*}
hence 
\begin{equation*}
\lim_{n\to+\8}\int_{\Omega_n\backslash \Omega} \dist^p(z,\pd \Omega_n)\d z=
\lim_{n\to+\8}\int_{\Omega \backslash \Omega_n} \dist^p(z,\pd \Omega)\d z=0.
\end{equation*}
 To prove
 \begin{equation*}
\lim_{n\to+\8}\int_{\Omega_n\cap \Omega}| \dist^p(z,\pd \Omega_n)-\dist^p(z,\pd \Omega)|\d z=0,
\end{equation*}
denote by $d_\H$ the Hausdorff distance, and
by the mean value theorem, it holds that
\begin{align*}
\int_{\Omega_n\cap \Omega}&| \dist^p(z,\pd \Omega_n)-\dist^p(z,\pd \Omega)|\d z\\
&\leq \int_{\Omega_n\cap \Omega}| \dist(z,\pd \Omega_n)-\dist(z,\pd \Omega)|
\cdot p \sup_{z\in \Omega_n\cap \Omega}\Big(\max \{\dist(z,\pd \Omega_n),
\dist(z,\pd \Omega)\}\Big)^{p-1}\d z\\
&\leq \H^2(\Omega_n\cap \Omega)d_\H(\pd \Omega_n,\pd \Omega)\cdot p C_1^{p-1}
\leq \pi C_1^2 d_\H(\pd \Omega_n,\pd \Omega)\cdot p C_1^{p-1}\to 0.
\end{align*}
Thus both terms \eqref{en-1} and \eqref{en-2} converge to zero, and \eqref{con3} is proven.

Combining with \eqref{con3} gives
\begin{equation*}
\E(\Omega)\leq \liminf_{n\to+\8}\E(\Omega_n)=\inf_{\-{\cal A}}\E,
\end{equation*}
 hence $\Omega$ is effectively a minimizer of $\E$. Since $\Omega$ is the limit of $\Omega_n \sse \mathcal{A}$, based on Lemma \ref{outA} and Corollary \ref{corOmega}, we have $\Omega \in \mathcal{A}$. 
 
\end{proof}

\section{Proof of regularity}\label{sec:proof}

This section completes the proof of Theorem~\ref{regularity} by establishing the desired
regularity of the minimizers. A few technical estimates used in the proof 
are left as separate lemmas proved at the end of the section.

\begin{proof}[\bf{Proof of Theorem~\ref{regularity}}]
Let $\Omega$ be a minimizer of $\E$, and
let $\gm$ be an arc-length parameterization of $\pd \Omega$.
 Assume there exist $M,\vep$, $t_1<t_2$ such that
\begin{eqnarray}
|\gm'(t_2)-\gm'(t_1)|=M\vep,\qquad t_2-t_1=\vep. \label{gammaprime}
\end{eqnarray}
The quantity $\vep$ is assumed to  be vanishingly small, and estimates
involving $\vep$ will be in general valid for sufficiently small $\vep$,
rather than all $\vep$.
The goal is to find an upper bound for $M$.

\medskip

Without loss of generality, upon rigid movements, we can assume $t_1=0$, $t_2=\vep$.
 Endow $\R^2$ with
a Cartesian coordinate system with
\begin{equation}
\gm(0)\in\{x\geq 0,y=0\} \qquad \gm(\vep)\in \{y\geq 0, x=0\},\qquad \gm'(0)=(0,1). \label{Cartesian}
\end{equation}
%
We first give an estimate on $\gm(\vep)_y$. Using H\"older's inequality, and recalling the fact that
$$\kappa_{\pd \Omega}\ll \H^1_{\llcorner \pd \Omega},\qquad \frac{\d\kappa_{\pd \Omega}}{\d \H^1_{\llcorner \pd \Omega}}\in L^2(0,\H^1(\pd 
\Omega);\R),$$
for any $t\in [0,\vep]$, it holds that
\begin{align*}
\frac{\E(\Omega)}\la&\geq\int_{\gm([0,t])} \kappa_{\pd \Omega}^2\d\H^1_{\llcorner \pd \Omega} 
=\int_0^t |\gm''|^2 \d s \geq \frac{|\gm'(t)-\gm'(0)|^2}{\vep} \geq \frac{|\gm'(t)_y-1|^2}{\vep}\\
&\Lra |\gm'(t)_y-1|  = 1-\gm'(t)_y\leq \sqrt{\vep \E(\Omega )/\la}\Lra \gm'(t)_y\ge 1-\sqrt{\vep \E(\Omega)/\la},
\end{align*}
hence
\begin{align*}
\gm(\vep)_y &= \int_0^\vep \gm'(t)_y\d t \ge  \int_0^\vep [1-\sqrt{\vep \E(\Omega)/\la}] \d t
=\vep[1-\sqrt{\vep \E(\Omega)/\la}]
\end{align*}
On the other hand, as
we imposed $\gm(\vep)_y\geq 0$, we have 
\[\gm(\vep)_y = |\gm(\vep)_y| =\bigg|\int_0^\vep \gm'(t)_y\d t\bigg| \le \int_0^\vep |\gm'(t)_y|\d t\le \vep.\]
In particular, 
\[\vep-\sqrt{ \E(\Omega)/\la}\vep^{3/2}\le  \gm(\vep)_y \le \vep,  \]
therefore, 
\begin{eqnarray} \label{gmepy}
\gm(\vep)_y=\vep+O(\vep^{3/2}).
\end{eqnarray}


\bigskip

Construct the competitor  
$\Omega_\vep$ in the following way:

\begin{enumerate}
\item denote by $t_\pm$ the two times such that $\gm'(t_\pm)=(\pm 1,0)$, and by
$t_\bot$ the time such that $\gm'(t_\bot)=(0,-1)$. Since we imposed $\gm'(0)=(0,1)$, without loss of generality we can assume
that the tangent direction turns {\em counterclockwise}, i.e., 
$$\vep< t_-<t_\bot<t_+<\H^1(\pd \Omega). $$
Note that 
\begin{align*}
\frac{2 }{t_\bot-t_-}&=
\frac{   |\gm'(t_\bot)-\gm'(t_-)|^2   }{t_\bot-t_-} \leq \int_{\gm([t_-,t_\bot])} \kappa_{\pd \Omega}^2\d\H^1_{\llcorner \pd \Omega} 
\leq \frac{\E(\Omega)}\la \overset{\eqref{inf E}}\leq 2 \la^{-1}+2\pi\\
&\Lra t_\bot-t_-\geq \frac{1}{ \la^{-1}+\pi }.
\end{align*}
Similarly, we get 
\begin{eqnarray} \label{lbd}
\min\{\H^1(\pd \Omega)-t_+,\ t_+-t_\bot,\ t_-\}\geq    \frac{1}{ \la^{-1}+\pi }.
\end{eqnarray}

\item Define the vector field
$v:[t_-,t_+]\lra \R^2$ as
\begin{align}
v(s):=\left\{
\begin{array}{cl}
\left( \cos\Big( \frac{\pi}{2}(1+\frac{s - t_-}{t_\bot-t_-})\Big) ,\sin\Big( \frac{\pi}{2}(1+\frac{s - t_-}{t_\bot-t_-})\Big)  \right),
& \text{if } s \in [t_-,t_\bot],\\
\left( \cos\Big( \frac{\pi}{2}(1+\frac{t_+ - s}{t_+ -t_\bot})\Big) ,-\Big(\frac{t_+ - s}{t_+-t_\bot}\Big)^2
\sin\Big( \frac{\pi}{2}(1+\frac{t_+ - s}{t_+-t_\bot})\Big)  \right),
& \text{if } s \in [t_\bot,t_+].
\end{array}
\right. \label{v}
\end{align}
Note first that $v$ is continuous (smooth outside $t_\bot$), and direct computation gives
 \begin{equation*}
v'(s)=\left\{
\begin{array}{cl}
\frac{\pi/2}{t_\bot-t_-}
\left( -\sin\Big( \frac{\pi}{2}(1+\frac{s - t_-}{t_\bot-t_-})\Big) ,\cos\Big( \frac{\pi}{2}(1+\frac{s - t_-}{t_\bot-t_-})\Big)  \right),
& \text{if } s \in [t_-,t_\bot),\\
\Big( \frac{\pi/2}{t_+-t_\bot}\sin\Big( \frac{\pi}{2}(1+\frac{t_+ - s}{t_+-t_\bot})\Big) 
,\frac{\pi/2}{t_+-t_\bot}\Big(\frac{t_+ - s}{t_+-t_\bot}\Big)^2\cos\Big( \frac{\pi}{2}(1+\frac{t_+ - s}{t_+-t_\bot})\Big)  
& \\
+2\frac{t_+ - s}{( t_+-t_\bot)^2 }\sin\Big( \frac{\pi}{2}(1+\frac{t_+ - s }{t_+-t_\bot})\Big)\Big)&\text{if } s \in (t_\bot,t_+].
\end{array}\right.
\end{equation*}
In particular, 
\begin{equation*}
\lim_{t \to t_\bot^-} v'(t) =\frac{\pi/2}{t_\bot-t_-}(0,-1),\qquad
\lim_{t  \to t_\bot^+} v'(t)= \frac{\pi/2}{t_+-t_\bot}(0,-1),
\end{equation*}
i.e., the left and right limit differ just by a multiplicative constant. This observation is crucial, since
it implies that  the tangent derivative of the arc-length reparameterization of $v$ does {\em not} jump 
at $t=t_\bot$ (recall also that $\gm'$ does not jump at $t=t_\bot$, hence the tangent derivative of the arc-length reparameterization
of $\gm+c v$ does not jump at $t=t_\bot$, for any $c>0$). 
We claim:
\begin{align}
\|v'\|_{L^\8}&\le  \max\Big\{ \frac{\pi/2}{t_\bot-t_-},\frac{4}{t_+-t_\bot} \Big\}
\leq 4 (\la^{-1} + \pi)    
<+\8,\label{t diff}\\
 \|v''\|_{L^\8}&\le   16(\la^{-1}+\pi)^2   
 <+\8.\label{t diff 2}
\end{align}
The proofs of both claims
are presented in Lemma \ref{v norm} below.
\item Let $\gm_\vep$ be the curve such that
\begin{equation}
\gm_\vep(t):=
\left\{
\begin{array}{cl}
(2\gm(t)_x, 2\gm(t)_y) & \text{if } t\in [0,\vep],\\
\gm(t)+(0,\gm(\vep)_y) & \text{if } t\in [\vep,t_-],\\
\gm(t)+\gm(\vep)_y v(t)&\text{if } t\in [t_-,t_+],\\
\left (
 \gm(t)_x\left(1+\frac{\gm(0)_x}{\gm(0)_x-\gm(t_+)_x}\right)
-\frac{\gm(t_+)_x\gm(0)_x}{\gm(0)_x-\gm(t_+)_x} 
 \right) & \text{if } t\in [t_+,\H^1(\pd \Omega)].
\end{array}
\right. \label{competitor}
\end{equation}
Note that $\gamma_\epsilon$ defined in \eqref{competitor} is injective.
Let $\pd \Omega_\vep$ be the image of $\gm_\vep$, and $\Omega_\vep$ be the bounded region of the plane delimited
by $\pd \Omega_\vep$. This will be our competitor. Observe first that, as $\gm'(t_+)=(1,0)$,
\begin{align*}
\lim_{t\to t_+^-} \gm_\vep'(t)=\left(1+\gm(\vep)_y\frac{\pi/2}{t_+-t_\bot}\right) (1,0),
\qquad \lim_{t\to t_+^+} \gm_\vep'(t)=\left(1+\frac{\gm(0)_x}{\gm(0)_x-\gm(t_+)_x} \right) (1, 0),
\end{align*}
i.e., the left and right limit differ just by a multiplicative constant. This observation is again crucial, since
it implies that  the tangent derivative of the arc-length reparameterization of $\gm+ \gm(\vep)_y v$ does {\em not} jump 
at $t=t_+$.
\end{enumerate}
Intuitively, for $t\in [0, H^1(\pd \Omega)] $ the competitor $\gm_\vep$ is constructed from $\gm$ by:
\begin{enumerate}
\item  a homothety of center $(0,0)$ and ratio 2 for $t\in[0,\vep] $,

\item a translation of the vector $(0,\gm(\vep)_y)$ for $t\in [\vep,t_-] $,

\item adding the smooth vector field $\gm(\vep)_y v(t)$ for $t\in [t_-,t_+]$,

\item a scaling of factor $1+\frac{\gm(0)_x}{\gm(0)_x-\gm(t_+)_x}$ and then translation to the right by $\frac{\gm(t_+)_x\gm(0)_x}{\gm(0)_x-\gm(t_+)_x} $ in the $x$ direction 
for $t\in [t_+,\H^1(\pd \Omega)] $.
\end{enumerate}
\begin{figure}[h!]
\begin{center}
\includegraphics[scale=1.0]{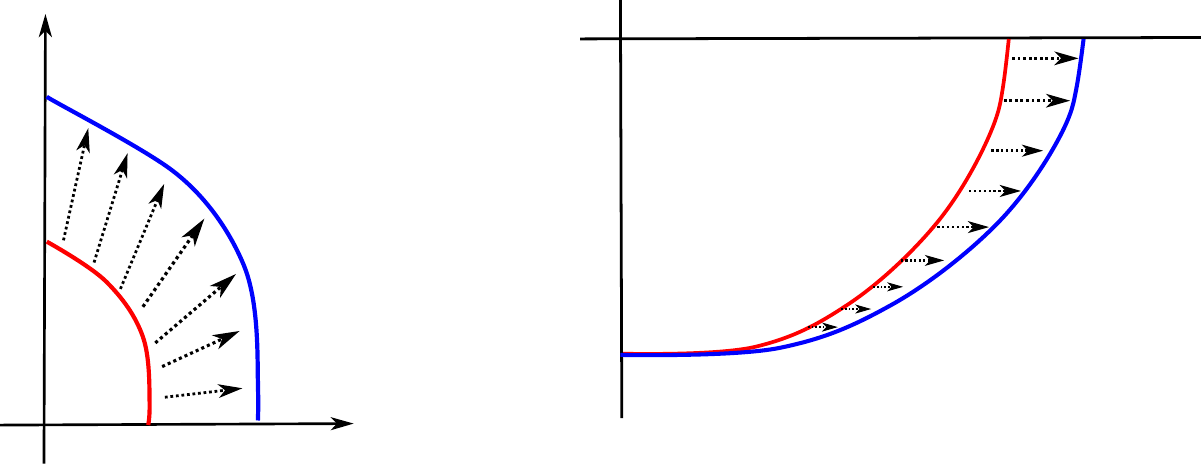}
\put(-342,0){$0$}\put(-252,0){$x$-axis}\put(-352,135){$y$-axis}
\put(-325,35){$\pd \Omega$}\put(-302,90){$\pd \Omega_\vep$}
\put(-313,0){\small{$\gm(0)$}}\put(-355,62){\small{$\gm(\vep)$}}
\put(-283,0){\small{$\gm_\vep(0)$}}\put(-357,102){\small{$\gm_\vep(\vep)$}}
\put(-155,125){$x$-axis}\put(-200,5){$\{x=\gm(t_+)_x\}$}
\put(-80,90){$\pd \Omega$}\put(-70,50){$\pd \Omega_\vep$}
\put(-85,130){$\substack{\gm(\H^1(\pd \Omega))\\ =\gm(0)}$}
\put(-33,110){$\substack{\gm_\vep(\H^1(\pd \Omega))\\ =\gm_\vep(0)}$}
\put(-200,30){$\substack{\gm(t_+)\\ =\gm_\vep(t_+)}$}
\caption{Representation of the construction of the competitor $\gm_\vep$,
 for $t\in [0,\vep]$ (left) and $t\in [t_+,\H^1(\pd \Omega)] $ (right).}
\end{center}
\end{figure}
It is straightforward to check compactness and convexity for $\Omega_\vep$. 
Moreover, denoting by $\~\gm_\vep$
the arc-length reparameterization of $\gm_\vep$, the curvature of $\~\gm_\vep$ is still a {\em function}
(instead of a more generic measure), as the is $\gm_\vep$ always constructed from $\gm$ via translation,
scaling, or sum with smooth vector fields, and the 
tangent derivative $\~\gm_\vep'$ never jumps at ``junction points'' (i.e., for $t=\vep,t_-,t_\bot,t_+,\H^1(\pd \Omega)$).
 
Next, to estimate
$\E(\Omega_\vep)-\E(\Omega)$, we claim
\begin{align}
\int_{\Omega_\vep}\dist^p(z,\pd \Omega_\vep) \d z-\int_{\Omega} \dist^p(z,\pd \Omega)\d z
  &\leq \vep \cdot p(C_1+1)^{p-1} \pi C_1^2/2  + (2 \vep)^{p+1} \pi (C_1 + 1).
\label{F}\\
 \int_{\pd \Omega_\vep} \kappa_{\pd \Omega_\vep}^2\d\H^1_{\llcorner \pd \Omega_\vep} -\int_{\pd \Omega} \kappa_{\pd \Omega}^2\d\H^1_{\llcorner \pd \Omega} 
 &\leq \vep\Big( C_2 -\frac{M^2 }2\Big)+O(\vep^{3/2}), \label{K}
\end{align}
with
\begin{equation*}
C_2 = 32 (\la^{-1}+ \pi)^2 + 32 \sqrt{2 C_1 \pi }  (\la^{-1} + \pi)^{5/2}
\end{equation*}
defined in \eqref{C_2}.

\medskip

{\em Step 1. Proof of \eqref{K}.} Using the notation from \eqref{competitor}, 
we make the following claims:
\begin{align}
\int_{\gm([\vep,t_-]) } \kappa_{\pd \Omega}^2\d\H^1_{\llcorner \pd \Omega} 
&= \int_{\gm_{\vep}([\vep,t_-]) } \kappa_{\pd \Omega_\vep}^2\d\H^1_{\llcorner \pd \Omega_\vep}, \label{K-1}\\
\int_{\gm_{\vep}([t_+,\H^1(\pd \Omega)]) } \kappa_{\pd \Omega_\vep}^2\d\H^1_{\llcorner \pd \Omega_\vep}-\int_{\gm([t_+,\H^1(\pd \Omega )]) } \kappa_{\pd \Omega}^2\d\H^1_{\llcorner \pd \Omega} 
&\le O(\vep^{3/2}) , \label{K-2}\\
 \int_{\gm_\vep([t_-,t_+]) } \kappa_{\pd \Omega_\vep}^2\d\H^1_{\llcorner \pd \Omega_\vep}
 -\int_{ \gm([t_-,t_+])} \kappa_{\pd \Omega}^2\d\H^1_{\llcorner \pd \Omega} 
&\leq   \vep  C_2 +O(\vep^{3/2}),\label{K-3}\\
\int_{\gm([0,\vep])} \kappa_{\pd \Omega }^2\d\H^1_{\llcorner \pd \Omega} 
- \int_{\gm_\vep([0,\vep])} \kappa_{\pd \Omega_\vep}^2\d\H^1_{\llcorner \pd \Omega_\vep} &\geq \frac{M^2 \vep}2.\label{K-4}
\end{align}
The proof of all four assertions are quite technical, and for reader's convenience, will be done in Lemmas \ref{K-1 proof}
and \ref{K-3 proof}
below.
Combining \eqref{K-1}, \eqref{K-2}, \eqref{K-3}, and \eqref{K-4} gives
\begin{equation*}
\int_{\Omega_\vep} \kappa_{\pd \Omega_\vep}^2 \d\H^1_{\llcorner \Omega_\vep} \leq 
\int_{\Omega } \kappa_{\pd \Omega }^2 \d\H^1_{\llcorner \Omega} + \vep\Big( C_2 -\frac{M^2}2\Big)+O(\vep^{3/2}),
\end{equation*}
hence \eqref{K}.

\bigskip

{\em Step 2. Proof of \eqref{F}.}
Recall that, the construction 
of the competitor $\Omega_\vep$ in \eqref{competitor} gives also  
\begin{enumerate}
\item[(1)] $\gm (\vep)_y > \gm(0)_x > 0 $.
\item[(2)]  For $t\in[\vep,\H^1(\pd \Omega)]$ the competitor $\gm_\vep (t)$ is obtained by translating
 $\gm(t)$ by a vector of length at most $2 \vep$. Moreover,
 it holds
 \[\frac{\gm(t)_x-\gm(t_+)_x}{\gm(0)_x-\gm(t_+)_x}\gm(0)_x  \le 2\vep,\qquad\fal
 t\in[\vep,\H^1(\pd \Omega)], \]
 since by construction we have $\gm(0)_x\le 2\vep$, and 
$\gm(0)$ is the point with the most positive $x$ coordinate,
 which ensures 
$\frac{\gm(t)_x-\gm(t_+)_x}{\gm(0)_x-\gm(t_+)_x}\le 1$.

\item[(3)] For $t\in [0,\vep]$, the competitor $\gm_\vep(t)$ is obtained by scaling $\gm(t)$
by a factor of 2.
\end{enumerate}
One readily checks for all $t$ it holds that $|\gm_\vep(t)-\gm(t)|\leq 2\vep$, and
$$d_\H(\pd \Omega_\vep,\pd \Omega)\leq 2\vep.$$
Thus, by the mean value theorem, for each point $x \in \Omega_\vep\cap \Omega$ it holds that
\begin{align*}
\dist^p(x,\pd \Omega_\vep) -\dist^p(x,\pd \Omega) &\leq (\dist(x,\pd \Omega_\vep) -\dist(x,\pd \Omega))  \\
&\qquad\cdot p
(\sup_{z \in \Omega_\vep\cap \Omega}\max\{\dist(x,\pd \Omega_\vep) ,\dist(x,\pd \Omega)\})^{p-1}\\
&\leq
2\vep  \cdot p(\diam(\Omega)+2\vep)^{p-1}
\overset{\eqref{diam-h}}\leq\vep \cdot 2p(C_1+1)^{p-1},
\end{align*}
 with $C_1$ defined in \eqref{C_1}. Thus, by convexity of $\Omega$,
 \begin{equation*}
\H^2(\Omega_\vep\cap \Omega) \leq \H^2( \Omega) \leq \pi (\diam(\Omega)/2)^2 \tcr{.}
\end{equation*}
 It follows that
 \begin{align}
\int_{\Omega_\vep\cap \Omega}\dist^p(x,\pd \Omega_\vep)\d x -\int_{\Omega}\dist^p(x,\pd \Omega)\d x &\leq 
\vep \cdot 2p(C_1+1)^{p-1}\H^2(\Omega_\vep\cap \Omega) \notag\\
&\leq  \vep \cdot p(C_1+1)^{p-1} \pi C_1^2/2. \label{av1}
\end{align}
Then note that, since by construction we have $d_\H(\pd \Omega_\vep,\pd \Omega)\leq 2 \vep$,
it follows that
\begin{eqnarray}\label{av2}
\int_{\Omega_\vep\backslash \Omega}\dist^p(x,\pd \Omega_\vep)\d x \leq (2\vep)^{p}\H^2(\Omega_\vep\backslash \Omega) & \leq  & (2 \vep)^{p} \cdot 2 \vep \H^1(\pd \Omega_{\vep}) \nonumber \\
 &\leq&   (2 \vep)^{p+1} \pi (\diam(\Omega) + 4 \vep) \leq  (2 \vep)^{p+1} \pi (C_1 + 1).
\end{eqnarray}
 The inequality $\H^2(\Omega_\vep\backslash \Omega) \le
	 2 \vep \H^1(\pd \Omega_{\vep}) $ is due to the convexity of $\Omega_\vep$,
	 the fact that $d_\H(\pd \Omega_\vep,\pd \Omega)\leq 2 \vep$, and
	 Lemma \ref{perimeter vs area of tubular neighborhood}.
Combining \eqref{av1} and \eqref{av2} gives
\begin{align}
\int_{\Omega_\vep}\dist^p(x,\pd \Omega_\vep)\d x  - \int_{\Omega}\dist^p(x,\pd \Omega)\d x &\leq  \vep \cdot p(C_1+1)^{p-1} \pi C_1^2/2  + (2 \vep)^{p+1} \pi (C_1 + 1).
\label{av}
\end{align}
Thus \eqref{F}
is proven.

\bigskip

Combining \eqref{F} and \eqref{K} we finally infer
\begin{align*}
\E(\Omega_\vep)&-\E(\Omega)\\
&=\int_{\Omega_\vep}\dist^p(x,\pd \Omega_\vep) \d x-\int_{\Omega} \dist^p(x,\pd \Omega)\d x
+ \la\bigg( \int_{\pd \Omega_\vep} \kappa_{\pd \Omega_\vep}^2\d\H^1_{\llcorner \pd \Omega_\vep}
 -\int_{\pd \Omega} \kappa_{\pd \Omega}^2\d\H^1_{\llcorner \pd \Omega}  \bigg)\\
&\leq   
\vep \cdot p(C_1+1)^{p-1} \pi C_1^2/2  + (2 \vep)^{p+1} \pi (C_1 + 1)
+\la\bigg(\vep\Big(C_2-\frac{M^2}2\Big)+O(\vep^{3/2})\bigg).
\end{align*}
Note also that the term $(2 \vep)^{p+1} \pi (C_1 + 1)  $ can be absorbed into $O(\vep^{3/2})$,
due to the condition $p\geq 1$, hence
\begin{equation*}
\E(\Omega_\vep)-\E(\Omega) \leq 
\la\bigg(\vep\Big( \la^{-1}p(C_1+1)^{p-1}\pi C_1^2/2 +C_2-\frac{M^2}2\Big)+O(\vep^{3/2})\bigg).
\end{equation*}
The minimality assumption on $\Omega$, and the arbitrariness of $\vep>0$ then imply
\begin{align*}
 & \la^{-1}p(C_1+1)^{p-1}\pi C_1^2/2 +C_2-\frac{M^2}2 \geq 0\\
&\Lra M\leq  \sqrt{ \la^{-1}p(C_1+1)^{p-1}\pi C_1^2+ 2 C_2} = C,
\end{align*}
and the proof is complete.

\end{proof}

\begin{lemma}\label{v norm}
Under the hypotheses of Theorem \ref{regularity}, assertions \eqref{t diff} and \eqref{t diff 2} hold.
\end{lemma}

\begin{proof}
We use the same notations from the proof of Theorem \ref{regularity}. Since
\begin{equation*}
v'(s) =\left\{
\begin{array}{cl}
\frac{\pi/2}{t_\bot-t_-}
\left( -\sin\Big( \frac{\pi}{2}(1+\frac{s - t_-}{t_\bot-t_-})\Big) ,\cos\Big( \frac{\pi}{2}(1+\frac{s - t_-}{t_\bot-t_-})\Big)  \right),
& \text{if } s \in [t_-,t_\bot),\\
\Big( \frac{\pi/2}{t_+-t_\bot}\sin\Big( \frac{\pi}{2}(1+\frac{t_+ - s }{t_+-t_\bot})\Big) 
,\frac{\pi/2}{t_+ - t_\bot}\Big(\frac{t_+ - s }{t_+-t_\bot}\Big)^2 \cos\Big( \frac{\pi}{2}(1+\frac{t_+ - s }{t_+-t_\bot})\Big)  
& \\
+2\frac{t_+ - s}{ (t_+-t_\bot)^2 }\sin\Big( \frac{\pi}{2}(1+\frac{t_+ - s}{t_+-t_\bot})\Big)\Big)&\text{if } s \in (t_\bot,t_+],
\end{array}\right.
\end{equation*}
it follows that
\begin{equation*}
|v'( s )|\le \frac{\pi/2}{t_\bot-t_-} \qquad \text{for any } s \in [t_-,t_\bot),
\end{equation*}
and
\begin{align*}
|v'(s)| &= \bigg[ \Big(\frac{\pi/2}{t_+-t_\bot}\Big)^2\sin^2\Big( \frac{\pi}{2}(1+\frac{t_+ - s}{t_+-t_\bot})\Big) 
+ \Big(\frac{\pi/2}{t_+-t_\bot}\Big)^2 \Big (\frac{t_+ - s }{t_+-t_\bot} \Big)^4 \cos^2\Big( \frac{\pi}{2}(1+\frac{t_+ - s }{t_+-t_\bot})\Big) \\
&\qquad+4 \frac{ (t_+ - s)^2 }{ (t_+-t_\bot)^4 } \sin^2\Big( \frac{\pi}{2}(1+\frac{t_+ - s}{t_+-t_\bot})\Big)  \\
&\qquad+4 \frac{\pi/2}{t_+-t_\bot}  \frac{ (t_+ - s)^3 }{( t_+-t_\bot )^4}  \cos\Big( \frac{\pi}{2}(1+\frac{t_+ - s }{t_+-t_\bot})\Big) 
\sin\Big( \frac{\pi}{2}(1+\frac{t_+ - s }{t_+-t_\bot})\Big) \bigg]^{1/2}\\
&\le \bigg[  \Big(\frac{\pi/2}{t_+-t_\bot}\Big)^2 + \frac{4 + \pi}{ ( t_+-t_\bot )^2 }\bigg]^{1/2}
\le \frac{4 }{t_+-t_\bot}
\end{align*}
for any $s \in (t_\bot,t_+]$. Thus \eqref{t diff} is proven.

\medskip

To prove \eqref{t diff 2}, note that for $s \in [t_-,t_\bot)$ it holds that
\begin{equation*}
v''(s)=-\bigg ( \frac{\pi/2}{t_\bot-t_-}\bigg ) ^2
\left( \cos\Big( \frac{\pi}{2}(1+\frac{s -t_-}{t_\bot-t_-})\Big) ,\sin\Big( \frac{\pi}{2}(1+\frac{s - t_-}{t_\bot-t_-})\Big)  \right),
\end{equation*}
hence
$|v''(s )| \le \bigg ( \dfrac{\pi/2}{t_\bot-t_-}\bigg )^2$. Similarly, for $s \in (t_\bot,t_+]$, it holds that
\begin{align*}
v''(s)&=\bigg( -\bigg ( \frac{\pi/2}{t_+-t_\bot}\bigg )^2\cos\Big( \frac{\pi}{2}(1+\frac{t_+ - s}{t_+-t_\bot})\Big) ,\\
&\qquad\frac{\pi/2}{t_+-t_\bot}\bigg[\frac{-2(t_+ - s)}{  (t_+-t_\bot )^2 }\cos\Big( \frac{\pi}{2}(1+\frac{t_+ - s }{t_+-t_\bot})\Big) 
+ \Big(\frac{t_+ - s }{t_+-t_\bot}\Big)^2 \frac{\pi/2}{t_+-t_\bot}\sin\Big( \frac{\pi}{2}(1+\frac{t_+ - s }{t_+-t_\bot})\Big) \bigg] \\
& 
\qquad -2\frac{t_+ - s }{ ( t_+-t_\bot)^2 }\frac{\pi/2}{t_+-t_\bot}\cos\Big( \frac{\pi}{2}(1+\frac{t_+ - s }{t_+-t_\bot})\Big)
-\frac{2}{ (t_+-t_\bot)^2 }\sin\Big( \frac{\pi}{2}(1+\frac{t_+ - s }{t_+-t_\bot})\Big)\bigg),
\end{align*}
therefore, 
using \eqref{lbd},
\begin{align*}
|v''(s)| & \le \bigg ( \frac{4}{t_+-t_\bot}\bigg )^2 
\le  16 (\la^{-1}+ \pi)^2,
\end{align*}
hence \eqref{t diff 2} is proven.
\end{proof}

The rest of the section contains the proofs of the technical estimates needed in the above proof.

\begin{lemma}\label{K-1 proof}
Under the hypotheses of Theorem \ref{regularity}, assertions \eqref{K-1}, \eqref{K-2} and \eqref{K-4} hold.
\end{lemma}

\begin{proof}
We use the same notations from the proof of Theorem \ref{regularity}.

\medskip

{\em Proof of \eqref{K-1}}. By construction, for any $t\in [\vep,t_-]$, $\gm_{\vep}(t)$ differ from
$\gm(t)$ by a translation, 
thus the curvature of these two segments are always equal, hence \eqref{K-1}. 

\medskip

{\em Proof of \eqref{K-2}}.  
For $t\in [t_+,\H^1(\pd \Omega)]$ we have
\begin{align*}
\gm_\vep(t)&=\left(  \gm(t)_x\left(1+\frac{\gm(0)_x}{\gm(0)_x-\gm(t_+)_x}\right)
-\frac{\gm(t_+)_x\gm(0)_x}{\gm(0)_x-\gm(t_+)_x} ,\gm(t)_y\right), \\
\gm_\vep'(t)&=\left(\gm'(t)_x\left(1+\frac{\gm(0)_x}{\gm(0)_x-\gm(t_+)_x}\right)
 ,\gm'(t)_y \right),\\
\gm_\vep''(t)&=\left(\gm '' (t)_x\left(1+\frac{\gm(0)_x}{\gm(0)_x-\gm(t_+)_x}\right)
 ,\gm '' (t)_y \right).
\end{align*}
%
%
%
%
We claim
\begin{align}
\gm(0)_x \le \int_0^\vep |\gm'(t)_x| \d t\leq \vep^{3/2} \sqrt{ \E(\Omega)/\la},\qquad
\gm(0)_x - \gm(t_+)_x \ge \frac\la{2\E(\Omega)} \label{K-2-2}.
\end{align}
In view of \eqref{Cartesian}, and noting that for any $t\in[0,\vep]$ it holds that
\begin{align*}
\frac{\E(\Omega)}\la&\geq\int_{\gm([0,t])} \kappa_{\pd \Omega}^2\d\H^1_{\llcorner \pd \Omega} 
=\int_0^t |\gm''|^2 \d s \geq \frac{|\gm'(t)-\gm'(0)|^2}{\vep} \geq \frac{|\gm'(t)_x|^2}{\vep}\\
&\Lra |\gm'(t)_x| \leq \sqrt{\vep \E(\Omega)/\la},
\end{align*}
it follows that
\begin{equation*}
|\gm(0)_x-\gm(\vep)_x|=|\gm(0)_x| \le \int_0^\vep |\gm'(t)_x| \d t\leq \vep^{3/2} \sqrt{ \E(\Omega)/\la}.
\end{equation*}
Now recall that by construction $\gm'(t_+)=(1,0)$,
$\gm'(\H^1(\pd \Omega))=\gm'(0)=(0,1)$, $|\gm'|\equiv 1$ for a.e. $t$, and
 let $\tau\in(t_+,\H^1(\pd \Omega))$ be time for which $\gm'(\tau)=(1/2,\sqrt{3}/2)$. Thus
\begin{align*}
\frac{\E(\Omega)}\la&\geq\int_{\gm([t_+,\tau])} \kappa_{\pd \Omega}^2\d\H^1_{\llcorner \pd \Omega} 
=\int_{t_+}^\tau |\gm''|^2 \d s \geq \frac{|\gm'(\tau)-\gm'(t_+)|^2}{\tau-t_+} = \frac1{\tau-t_+}\\
&\Lra  \tau-t_+\geq \la/\E(\Omega),
\end{align*}
and since $\gm_x'\geq 0$ on $[t_+,\H^1(\pd \Omega)]$, and $\gm_x'\geq 1/2$ for all $t\in [t_+,\tau]$, it follows that
$\gm(0)_x-\gm(t_+)_x \geq  \gm(\tau)_x-\gm(t_+)_x \geq   \frac{\la}{2\E(\Omega)}$, hence \eqref{K-2-2} is proven.
Consequently,
\begin{equation*}
\bigg| \frac{\gm(0)_x}{\gm(0)_x - \gm(t_+)_x} \bigg| \leq 2(\vep \E(\Omega)/\la)^{3/2} =O(\vep^{3/2}).
\end{equation*}
Therefore,
\begin{align*}
|\gm_\vep'|^{-4} &= \left(|\gm_x'|^2 \Big(1+\frac{\gm(0)_x}{\gm(0)_x-\gm(t_+)_x}\Big)^2+
 |\gm_y'|^2\right)^{-2}\\
 &=\left(1+|\gm_x'|^2 \frac{2\gm(0)_x}{\gm(0)_x-\gm(t_+)_x} +|\gm_x'|^2 \bigg ( \frac{\gm(0)_x}{\gm(0)_x - \gm(t_+)_x} \bigg )^2 
\right)^{-2} = 1+ O(\vep^{3/2}).
\end{align*}

Observe that for $t\in [t_+,\H^1(\pd \Omega)]$ we have
\begin{eqnarray*}
\kappa_{\pd \Omega_\vep} = \frac{\left |  \frac{\gm_\vep ''}{ | \gm_\vep' |  }  -  \gm_\vep'  \frac{  \langle \gm_{\vep}'', \gm_\vep' \rangle }{ | \gm_\vep' |^3}    \right |}{ | \gm_\vep' |},
\end{eqnarray*}
\begin{eqnarray*}
\langle \gm_{\vep}'', \gm_\vep'  \rangle = \gm_x' \gm_x''  \left(  1+\frac{\gm(0)_x}{\gm(0)_x-\gm(t_+)_x}   \right)^2 + \gm_y' \gm_y'' = \langle \gm', \gm'' \rangle + O(\vep^{3/2}) = O(\vep^{3/2}).
\end{eqnarray*}
 Here we use the fact that
\begin{equation}
\langle \gm'' ,\gm' \rangle =\frac12\frac{\d}{\d t}|\gm'|^2=0, \label{est1}
\end{equation}
since $\gm$ is parameterized by arc-length.
And we have
\begin{align*}
\int_{t_+}^{\H^1(\pd \Omega)} \frac{|\gm_\vep''|^2}{|\gm_\vep'|^{4}} \d t 
&=\int_{t_+}^{\H^1(\pd \Omega)}  \bigg(|\gm''|^2+|\gm_x''|^2\frac{2\gm(0)_x}{\gm(0)_x-\gm(t_+)_x} +|\gm_x''|^2  \left (\frac{2\gm(0)_x}{\gm(0)_x-\gm(t_+)_x} \right )^2 \bigg)(1+ O(\vep^{3/2})) \d t \\
&=\int_{t_+}^{\H^1(\pd \Omega)}  |\gm''|^2 \d t +O(\vep^{3/2}) ,
\end{align*}
hence \eqref{K-2}.

\medskip

{\em Proof of \eqref{K-4}}. In the time interval $[0,\vep]$, the competitor is obtained by scaling  by a factor of 2,
and direct computations give that the integrated squared curvature scales by a factor of $1/2$. 
Thus based on \eqref{gammaprime} we get
\begin{align}
\int_{0}^\vep \bigg|\frac{\d}{\d t}\bigg(\frac{\gm'}{|\gm'|}\bigg)\bigg|^2\ 
dt
&-
 \frac{1}{2}  \int_{0}^{\vep} \bigg|\frac{\d}{\d t}\bigg(\frac{\gm_\vep'}{|\gm_\vep'|}\bigg)\bigg|^2\
dt
 \notag\\
 &  =
\frac12   \int_{0}^\vep \bigg|\frac{\d}{\d t}\bigg(\frac{\gm'}{|\gm'|}\bigg)\bigg|^2\
dt 
= \frac12   \int_{0}^\vep \left |\gm'' \right |^2
\geq\frac{M^2}2\vep ,
\label{tvep}
\end{align}
hence \eqref{K-4}.
\end{proof}

\begin{lemma}\label{K-3 proof}
Under the hypotheses of Theorem \ref{regularity}, assertion \eqref{K-3} holds.
\end{lemma}

\begin{proof}
We use the same notations from the proof of Theorem \ref{regularity}.
 In the time interval $[t_-,t_+]$, $\gm_\vep$ is given by
\begin{equation*}
\gm_\vep(t)=\gm(t)+\gm(\vep)_y v(t) ,\qquad t\in [t_-,t_+].
\end{equation*}
Note first that since $\Omega$ is a minimizer of $\E$, it must hold
$$ \int_{ \gm([t_-,t_+])} \kappa_{\pd \Omega}^2\d\H^1_{\llcorner \pd \Omega}  <+\8,$$
and recalling our definition of integrated squared curvature in \eqref{K2}, it follows that the Radon-Nikodym derivative
$\frac{\d\kappa_{\pd \Omega}}{\d\H^1_{\llcorner \pd \Omega}} $ is square integrable. In terms of the parameterization
$\gm_\vep$, this gives
$$ 
 \frac{1 }{|\gm_\vep'|}  
\frac{\d}{\d t}\bigg(\frac{\gm_\vep'}{|\gm_\vep'|} \bigg)=
 \frac{1 }{|\gm_\vep'|}   \left (
\frac{\gm_\vep''}{|\gm_\vep'|} - \gm_\vep'\frac{\langle \gm_\vep'',\gm_\vep'\rangle}{|\gm_\vep'|^3} \right )
 \in L^2(0,\H^1(\pd \Omega);\R).$$
Recall \eqref{gmepy}, that is, $\gm(\vep)_y=\vep+O(\vep^{3/2})$, and
as $\gm$ is parameterized by arc-length (i.e., $|\gm' |=1$ for a.e. $t$), and $v$ was defined
in \eqref{v} (in particular, $|v'|$ was uniformly bounded from above), it follows that
\begin{equation*}
|\gm_\vep' | = \sqrt{1+2\vep\langle \gm',v'\rangle +O(\vep^{3/2})}.
\end{equation*}
Then, for any $\al\in \R$ and sufficiently small $\vep$, we have 
\begin{equation}
|\gm_\vep' |^\al = 1+\al\vep\langle \gm',v'\rangle +O(\vep^{3/2}).\label{power}
\end{equation}
Calculation shows that
\begin{align*}
\frac{1 }{|\gm_\vep'|} 
\frac{\d}{\d t}\bigg(\frac{\gm_\vep'}{|\gm_\vep'| } \bigg)&=
\frac{\gm_\vep''}{|\gm_\vep'|^2 } - \gm_\vep'\frac{\langle \gm_\vep'',\gm_\vep'\rangle}{|\gm_\vep'|^4} \\
&=
\frac{\gm''+\vep v''}{|\gm_\vep'|^2 }
-\frac{\gm_\vep'}{|\gm_\vep'|^4 }\big(\langle \gm'' ,\gm' \rangle 
+\vep^2\langle  v'', v'\rangle+ \vep\langle  \gm'', v'\rangle+ \vep\langle  \gm', v''\rangle\big) + \text{ higher order terms}.
\end{align*}
Based on \eqref{est1}, we observe:
\begin{enumerate}

\item As both $|v'|$ and $|v''|$ are uniformly bounded from above, the term $\vep^2\langle  v'', v'\rangle$ 
is of order $O(\vep^{2})$.

\item The norm of $\vep\gm_\vep' \langle  \gm', v''\rangle/|\gm_\vep'|^4$ is estimated by
\begin{equation}
\vep \bigg|\frac{\gm_\vep' \langle  \gm', v''\rangle}{|\gm_\vep'|^4 }\bigg| 
\leq \vep \frac{|\gm'|\cdot|v''|}{|\gm_\vep'|^3 } \leq \vep\|v''\|_{L^\8}+O(\vep^{2}).\label{est2}
\end{equation}

\item The norm of $\vep\gm_\vep' \langle  \gm'', v'\rangle/|\gm_\vep'|^4 $ is estimated by
\begin{equation}
\vep \bigg|\frac{\gm_\vep' \langle  \gm'', v'\rangle}{|\gm_\vep'|^4}\bigg| 
\leq \vep \frac{|\gm''|\cdot|v'|}{|\gm_\vep'|^3}. \label{est3}
\end{equation}
\end{enumerate}
Thus combining \eqref{est1}, \eqref{est2} and \eqref{est3} gives
\begin{align*}
\int_{ t_-}^{t_+}
\bigg| \gm_\vep'&\frac{\langle \gm_\vep'',\gm_\vep'\rangle}{|\gm_\vep'|^4} \bigg|^2\d t
=
\int_{ t_-}^{t_+} \bigg| \frac{\gm_\vep'}{|\gm_\vep'|^4 }\big(\langle \gm'' ,\gm' \rangle 
+\vep^2\langle  v'', v'\rangle+ \vep\langle  \gm'', v'\rangle+ \vep\langle  \gm', v''\rangle\big)\bigg|^2\d t  + \text{ higher order terms }\\
&\leq 
\int_{ t_-}^{t_+} |\gm_\vep'|^{-6}\big|
 \vep\langle  \gm'', v'\rangle+ \vep\langle  \gm', v''\rangle\big|^2\d t +O(\vep^{3})\\
 &\leq
2 \int_{ t_-}^{t_+} |\gm_\vep'|^{-6}
 (|\vep\langle  \gm'', v'\rangle|^2+| \vep\langle  \gm', v''\rangle|^2)\d t +O(\vep^{3})\\
&=2\vep^2 \int_{ t_-}^{t_+} |\gm_\vep'|^{-6}
\left (  | \langle  \gm'', v'\rangle|^2  +  | \langle  \gm', v''\rangle|^2   \right ) \d t +O(\vep^{3}) \\
& \leq 2\vep^2 \|v'\|_{L^\8}^2\int_{ t_-}^{t_+} |\gm_\vep'|^{-6}
| \gm'' |^2\d t +O(\vep^{2}).
\end{align*}
In view of \eqref{power}, we get 
\begin{equation*}
\int_{ t_-}^{t_+} |\gm_\vep'|^{-6} 
| \gm'' |^2\d t \leq 2 \int_{ t_-}^{t_+} 
| \gm'' |^2\d t \leq 2\int_{\pd \Omega}\kappa_{\pd \Omega}^2\d\H^1_{\llcorner \partial \Omega}<+\8,
\end{equation*}
hence 
\begin{equation*}
2\vep^2 \|v'\|_{L^\8}^2\int_{ t_-}^{t_+} |\gm_\vep'|^{-6}
| \gm'' |^2\d t \le O(\vep^{2}).
\end{equation*}
Thus 
\begin{equation*}
\int_{ t_-}^{t_+}
\bigg| \gm_\vep'\frac{\langle \gm_\vep'',\gm_\vep'\rangle}{|\gm_\vep'|^4 } \bigg|^2\d t \le O(\vep^{2}),
\end{equation*}
 and
\begin{align*}
\int_{ \gm([t_-,t_+])} \kappa_{\pd \Omega_\vep}^2\d\H^1_{\llcorner \pd \Omega_\vep} 
&= \int_{ t_-}^{t_+} \bigg|  \frac{1}{|\gm_\vep'|}   \frac{\d}{\d t}\bigg(\frac{\gm_\vep'}{|\gm_\vep'|} \bigg)\bigg|^2\d t
=
\int_{ t_-}^{t_+}
\bigg|\frac{\gm_\vep''}{|\gm_\vep'|^2} - \gm_\vep'\frac{\langle \gm_\vep'',\gm_\vep'\rangle}{|\gm_\vep'|^4} \bigg|^2\d t\\
&=\int_{ t_-}^{t_+}
\bigg|\frac{\gm_\vep''}{|\gm_\vep'|^2} \bigg|^2\d t+O(\vep^{2}).
\end{align*}
Again, in view of \eqref{power},
it follows that
\begin{align}
\int_{ t_-}^{t_+}
&\bigg|\frac{\gm_\vep''}{|\gm_\vep'|^2} \bigg|^2\d t = \int_{ t_-}^{t_+} \bigg|\frac{\gm''+\vep v''}{|\gm_\vep'|^2 } \bigg|^2\d t \notag \\
&=\int_{ t_-}^{t_+} (\langle \gm''+\vep v'',\gm''+\vep v'' \rangle)(1-4\vep\langle \gm',v'\rangle +O(\vep^{3/2})) \d t  + \text{ higher order terms } \notag\\
&=\int_{ t_-}^{t_+} (|\gm''|^2+2\vep\langle \gm'', v'' \rangle +\vep^2|v''|)(1-
4\vep\langle \gm',v'\rangle +O(\vep^{3/2})) \d t    + \text{ higher order terms } \notag\\
&\le
(1+4\vep\|v'\|_{L^\8})\int_{ t_-}^{t_+} |\gm''|^2 \d t + 2\vep \|v''\|_{L^\8}\int_{ t_-}^{t_+}  |\gm''| \d t
 +O(\vep^{3/2}) \notag
\end{align}
\begin{align}
&\le
(1+4\vep\|v'\|_{L^\8})\int_{ t_-}^{t_+} |\gm''|^2 \d t + 2\vep \|v''\|_{L^\8}  \bigg(\H^1(\pd \Omega)
\int_{ t_-}^{t_+}  |\gm''|^2 \d t\bigg)^{1/2}
+ O(\vep^{3/2}) \notag\\
&\leq
\int_{ t_-}^{t_+} |\gm''|^2 \d t  +4\vep \|v'\|_{L^\8}\E(\Omega)/\la+  2\vep \|v''\|_{L^\8} \sqrt{\H^1(\pd \Omega)\E(\Omega)/\la}
+ O(\vep^{3/2}). \label{in1} 
\end{align}
Note that
\begin{align*}
 4 \|v'\|_{L^\8}\E(\Omega)/\la&+  2 \|v''\|_{L^\8} \sqrt{\H^1(\pd \Omega)\E(\Omega)/\la} \\
&\le  
32 (\la^{-1}+ \pi)^2 + 32 \sqrt{2 C_1 \pi }  (\la^{-1} + \pi)^{5/2}
= C_2
\end{align*}
in view of \eqref{C_2}, \eqref{diam-h}, \eqref{t diff}, \eqref{t diff 2}, and Lemma \ref{perimeter vs diameter convex sets}.
Hence \eqref{K-3} follows from \eqref{in1}.
\end{proof}

\section{Conclusion}\label{sec:conclusion}

In this paper we investigated the minimization problem for the average distance functional defined for a two-dimensional domain with respect to its boundary, subject to a penalty proportional to
the Euler elastica of the boundary. We proved the existence and $C^{1,1}$-regularity of minimizers, mainly relying on the  method of contradictions by constructing suitable competitors. Echoing the large amount of existing studies that 
have exclusively focused on either the 1D average distance problem or the 2D Willmore energy question, by considering  variational problems associated with  combined energy functional, this study enriches and deepens our understanding of penalized average distance problem. Questions on the exact shape of a minimizer are still open and worth investigating in future. Limiting behaviors of the minimizers, property scaled with respect to $\lambda$, as  $\lambda\to 0$ and $\lambda\to \infty$
may also shed light to this class of free boundary problems.

\vskip 1cm

\noindent{\Large \bf Acknowledgement} The authors would like to thanks the referees for their careful reading  and for their valuable suggestions to improve the quality of the manuscript.

\appendix
\renewcommand{\theequation}{A.\arabic{equation}}
\renewcommand{\thelemma}{A.\arabic{lemma}}
\setcounter{equation}{0}

\vskip 1cm

\noindent{\Large \bf Appendix A}
\\

	Here we collect some results about convex sets, convergence in $\A$,
	and their effect on geometric quantities, such as perimeter, diameter, etc..
	One elementary yet crucial observation is that, given a two-dimensional convex set $\om\in \A$,
	then every point $x\in \A$ has a set $U\sse \om$ of {\em positive} area 
	containing $x$.
	
\begin{lemma}\label{perimeter vs diameter convex sets}
 {\cite[p. 1]{SG}} Let $n \geq 2$ and let $A,  B \subset \mathbb{R}^n$  be two convex bodies (i.e., compact convex sets with non-empty interior). If $A \subset B$, then the monotonicity of perimeters holds, i.e.
\begin{eqnarray}
\H^{n-1} (\pd A)\leq \H^{n-1} (\pd B ).
\end{eqnarray}
As a consequence, since any set with diameter $d$ is contained
in a ball of diameter $d$, we have
\begin{equation*}
\H^1(\pd \Omega)\leq \pi \diam (\Omega), \qquad \text{for all } \Omega\in \A.
\end{equation*}
\end{lemma}

\begin{lemma}
	Consider a sequence $\om_n \sse \A$, 
 converging to $\om\in \mathcal{A}$ in the topology
	of $\mathcal{A}$, such that $\bigcup_n \om_n \sse K	$
	for some compact set $K$,.
	Then $\diam (\om_n)\to \diam (\om)$.
\end{lemma}

\begin{proof}
	Consider a sequence $\om_n \sse \A$, converging to $\om$ in the topology
	of $\mathcal{A}$.  
	 As $\om_n$ are compact, there exist $x_n,y_n \in \om_n$
	such that $|x_n-y_n|=\diam (\om_n)$. By our assumption that
	 $\om_n \sse \A$ are all contained in a given
	compact set $K$, we have, up to a subsequence, $x_n \to x$, $y_n \to y$
	for some $x,y\in \om$. Thus it is clear that
	\[ \diam (\om_n)=|x_n-y_n| \to |x-y|\le \diam (\om). \]
	We need to exclude the strict inequality case.
	This is achieved by a contradiction argument: assume the opposite, i.e.
	there exist $v,w\in \om$ such that
	$|v-w|>|x-y|$. Then, we claim that there exist sequences $v_n,w_n$ of points
	in $\om_n$ such that, up to a subsequence, $v_n\to v$, $w_n\to w$. This, because of
	the opposite, would leave the existence of some set $U\sse \om$
	of positive area, containing either
	$ v$ or $w$, such that there are no sequence of points in $\om_n$ that
	enter into $U$. This contradicts the fact that 
	$\om_n $ is converging to $\om$ in the topology
	of $\mathcal{A}$. The proof is thus complete.  
\end{proof}

\begin{lemma}\label{perimeter vs area of tubular neighborhood}
	Given convex sets $\Omega_\vep, \Omega$ such that 
	$d_\H(\pd \Omega_\vep,\pd \Omega)\leq \dt$, it holds
	\begin{eqnarray}
\label{perimeter vs area of tubular neighborhood equation}
	\H^2(\Omega_\vep\backslash \Omega)\le 2\dt\H^1(\pd \Omega_{\vep}). 
	\end{eqnarray}

\end{lemma}

\begin{proof}
	Clearly, $\Omega_\vep\backslash \Omega$ is entirely contained
	in 
	\[\{x\in \Omega_\vep: \dist(x,\Omega_\vep)\le \dt \},\]
	i.e. the part of the tubular neighborhood of $\pd \Omega_{\vep}$ with thickness
	$\dt$ that lies inside $\Omega_\vep$. 
	
	We now use an approximation argument: we approximate $\pd \Omega_{\vep}$ with 
	convex polygons $P_n\sse \Omega_{\vep}$,
	e.g. by choosing $n$ point $x_{1,n},\cdots, x_{n,n}\in \pd \Omega_{\vep}$ 
	such that $\sup_{i}|x_{i+1,n}-x_{i,n}|\le 2\H^1(\pd \Omega_{\vep})/n $,
	and then connecting $x_{i+1,n}$ to $x_{i,n}$ with line segments.
	
	Note that the area of the difference is continuous
	with respect to such an approximation, i.e. 
	$\H^2(P_n\backslash \Omega) \nearrow \H^2(\Omega_\vep\backslash \Omega)$.
	It is then a straightforward computation to check that 
		\[\H^2( \{x\in P_n: \dist(x,P_n)\le \dt \} ) \le 2\dt \H^1(\pd P_n) ,\]
		for all sufficiently large $n$.
	\end{proof}

\vspace{ 0.5cm}

\end{document}